\documentclass[12pt,twoside]{amsart}
\usepackage{amsthm,amsfonts,amssymb,amscd,euscript}

\usepackage[matrix,arrow]{xy}

\author{Florin Ambro} 
\address{Institute of Mathematics ``Simion Stoilow'' of the Romanian
Academy\\
P.O. BOX 1-764, RO-014700 Bucharest\\ 
Romania.}
\email{florin.ambro@imar.ro}

\newcommand{\isoto}{{\overset{\sim}{\rightarrow}}}

\newcommand{\Q}{{\mathbb Q}}
\newcommand{\Z}{{\mathbb Z}}
\newcommand{\N}{{\mathbb N}}
\newcommand{\R}{{\mathbb R}}

\newcommand{\bP}{{\mathbb P}} 
\newcommand{\bA}{{\mathbb A}} 

\newcommand{\cB}{{\mathcal B}}

\newcommand{\cF}{{\mathcal F}}
\newcommand{\cG}{{\mathcal G}}
\newcommand{\cH}{{\mathcal H}}

\newcommand{\cX}{{\mathcal X}}
\newcommand{\cZ}{{\mathcal Z}}

\newcommand{\cL}{{\mathcal L}}
\newcommand{\cO}{{\mathcal O}}

\newcommand{\fm}{{\mathfrak m}}

\newcommand{\Bs}{\operatorname{Bs}}

\newcommand{\codim}{\operatorname{codim}}
\newcommand{\Conv}{\operatorname{Conv}}

\newcommand{\dv}{\operatorname{div}}
\newcommand{\emb}{\operatorname{emb}}

\newcommand{\Fix}{\operatorname{Fix}}

\newcommand{\Int}{\operatorname{int}}

\newcommand{\lct}{\operatorname{lct}}

\newcommand{\mld}{\operatorname{mld}}

\newcommand{\mult}{\operatorname{mult}}

\newcommand{\relint}{\operatorname{relint}}

\newcommand{\Sing}{\operatorname{Sing}}

\newcommand{\Supp}{\operatorname{Supp}}

\newcommand{\vol}{\operatorname{vol}}

\newcommand{\width}{\operatorname{width}}

\theoremstyle{plain}
\newtheorem{thm}{Theorem}[section]

\newtheorem{question}[thm]{Question}
\newtheorem{lem}[thm]{Lemma}
\newtheorem{cor}[thm]{Corollary}

\newtheorem{prop}[thm]{Proposition}

\theoremstyle{definition}

\newtheorem{exmp}[thm]{Example}
\newtheorem{exmps}[thm]{Examples}

\theoremstyle{remark}

\begin{document}

\bibliographystyle{amsalpha+}
\title{Variation of log canonical thresholds in linear systems}
\maketitle

\begin{center}
{\em Dedicated to the memory of Professor \c{S}erban Basarab}
\end{center}

\begin{abstract} 
We investigate the variation of log canonical thresholds in (graded) linear systems. 
For toric log Fano varieties, we give a sharp lower bound
for log canonical thresholds of the anticanonical members in terms of the 
global minimal log discrepancy.
\end{abstract}



\footnotetext[1]
{
This work was supported by a grant of the 
Romanian National Authority for Scientific Research, 
CNCS - UEFISCDI, project number PN-II-RU-TE-2011-3-0097.
}



\section*{Introduction}


An interesting question in the classification theory of algebraic varieties is the 
conjecture of Alexeev and Borisov brothers: (singular) Fano varieties $X$
of fixed dimension $d$ and with minimal log discrepancy $\mld(X)\ge \epsilon>0$, 
belong to a bounded family. This conjecture is known if $d=2$~\cite{Ale94} or 
$X$ is a toric variety~\cite{BB92}. To a Fano variety $X$ we associate the 
$\alpha$-invariant 
$$
\gamma(X)=\inf\{\lct(X;\frac{D_n}{n});n\ge 1,D_n\in |-nK_X|\}.
$$
It is known that $\gamma(X)\cdot \sqrt[d]{(-K_X)^d}\le d$, so the anticanonical volume 
is bounded above if the $\alpha$-invariant is bounded away from zero. It follows that the above conjecture 
reduces to two local statements: a) a lower bound $\gamma(X)\ge \gamma(d,\epsilon)>0$,
and b) an upper bound $r\le r(d,\epsilon)$ for the smallest integer $r\ge 1$ such that $rK_X$ is Cartier.
The first aim of this paper is to propose a sharp lower bound $\gamma(d,\epsilon)$,
and to establish it in the toric case.

\begin{thm}\label{MT}
Let $(X,B)$ be a toric log Fano variety, with $\dim X=d$ and 
$\mld(X,B)\ge \frac{1}{q}$, for some integer $q\ge 1$. Then 
$$
\gamma(X,B)=\inf\{\lct(X,B;\frac{D_n}{n});n\ge 1,D_n\in |-nK_X-nB |\}\ge \frac{q}{u_{d+1,q}},
$$
where $(u_{p,q})_{p,q\ge 1}$ is the sequence of integers defined recursively by 
$u_{1,q}=q,u_{p+1,q}=u_{p,q}(1+u_{p,q})$.
\end{thm}

Theorem~\ref{MT} is sharp (see Example~\ref{mE}). An interesting feature is that the
type of coefficients of $B$ do not matter. We expect the same bound holds in the non-toric
case. This is easy to see in dimension one, but unclear in dimension two.

The $\alpha$-invariant can be localized, and defined for polarized log varieties $(X,B;H)$. 
One defines $\gamma(X,B;H)=\inf_{P\in X}n\cdot \gamma_P(X,B;|nH|)$,
where 
$
\gamma_P(X,B;|nH|)=\inf\{\lct_P(X,B;D);D\in |-nH|\}.
$
The second aim of the paper is to study the variation of $\lct_P(X,B;D)$ in the variables $P,D$,
as suggested in~\cite{Amb06}.
The variation in $D$ is implicit in the work of Viehweg~\cite[Sections 5.3,8.2]{Vie95}. He considers only the 
special case when $B=0$ and $X$ has klt singularities, and replaces $\gamma(X,0;D)$ by $e(X,0;D)$, 
the largest positive integer $e$ such that $\frac{1}{e}<\lct(X,0;D)$. We extend Viehweg's results 
to our setting. Allowing non-zero boundaries strengthens some statements, and gives new ones. 
The main new statement is 

\begin{thm}
Let $\pi\colon (X,B)\to S$ be a family of log varieties. Then the locus of points $x\in X$ where 
$(X_{\pi(x)},B_{\pi(x)})$ has log canonical singularities is open.
\end{thm} 

It follows that log canonical thresholds are lower semi-continuous in flat families. In particular,
$\lct_P(X,B;D)$ is lower semicontinuous in two variables $(P,D)\in X\times |nH|$.
By projection onto the second factor, we obtain extensions to the log category of the semi-continuity 
results of Varchenko~\cite{Var82} and Demailly-Koll\'ar~\cite{DemKol01}. By projection onto the 
first factor, we obtain that $P\mapsto \gamma_P(X,B;|nH|)$ is lower semi-continuous and takes 
only finitely many values. It would be interesting to find out if the same holds for the asymptotic 
version $P\mapsto \gamma_P(X,B;H)$. 

We outline the structure of this paper. In Section 1 we introduce families of log varieties. The main
result is the openness of the locus where the fibers have log canonical singularities (Theorem~\ref{OpenLc}).
It follows that log canonical thresholds are semi-continuous in flat families (Corollary~\ref{8.5}).

In Section 2, we study the variation of log canonical thresholds $\lct_P(X,B;D)$ in $P$ and $D$,
where $D$ moves in a linear system. The results follow from those of Section 1 applied to the 
universal divisor of the linear system. We also generalize the product theorem of Viehweg to 
distinct factors (Theorem~\ref{pt}). To an $\N$-graded convex family of linear systems $\Lambda_\bullet$ 
on a given log variety $(X,B)$, we associate the local $\alpha$-invariant at $x$ as follows:
$$
\gamma_x(X,B;\Lambda_\bullet)=\inf\{\lct_x(X,B;\frac{D_n}{n}); n\ge 1,D_n\in \Lambda_n\}
$$
It is very interesting to study the variation in $x$ of this functional (see Question~\ref{grsc}),
but we can only say little in general. We compute this invariant if $X$ is a curve, recall
some known results, and express it in terms of width. We can say more in the toric case.
In Section 3, we compute the local $\alpha$-invariant in the generic point of the invariant
primes of a toric variety, and show that their minimum is exactly the global $\alpha$-invariant 
(Theorem~\ref{mSa}). In particular, we can compute combinatorially the $\alpha$-invariant of
a line bundle on a toric variety. With some extra assumptions, this was independently obtained
in~\cite[Theorem 3.4]{Dcx14}, with analytic methods.

The main result of Section 4 is Theorem~\ref{LHN}, a logarithmic effective version of a
diophantine approximation result of Hensley~\cite[Lemma 2]{Hen83}. The special case $q=1$ 
was solved by Averkov~\cite[Theorem 2.1]{Ave12}, and our proof is inspired from his. We 
also give sharp versions of the original results of Hensley~\cite{Hen83}.

In Section 5, we give the sharp lower bound for the $\alpha$-invariant of a toric log Fano
(Theorem~\ref{GB}), a result essentially equivalent to Theorem~\ref{LHN}. In particular,
we obtain an upper bound for the anti log canonical volume (Corollary~\ref{vb}).
Theorem~\ref{FLF} extends to the log category and simplifies the proof of the 
finiteness of $d$-dimensional $\epsilon$-log canonical toric Fano
varieties~\cite{BB92}. We end Section 5 with the simplest examples of toric log Fano varieties,
where we can explicitly compute both the minimal log discrepancy and the $\alpha$-invariant,
and see their relation with diophantine approximation.


\section{Families of log varieties}



\subsection{Relative effective Cartier divisors}


Throughout this paper, we fix a base field $k$, algebraically closed and of 
characteristic zero. By {\em scheme} we mean a $k$-scheme of finite type.

Let $\pi\colon X\to S$ be a flat morphism of schemes, and $D$ an effective Cartier
divisor on $X$. Recall that $D$ is called {\em relative over $S$} if one of the following
equivalent conditions holds:
\begin{itemize}
\item[a)] $D$ is flat over $S$.
\item[b)] for every point $x\in X$ and a local equation $f$ for $D$ at $x$, $f$ does not divide
zero in $\cO_{X_{\pi(x)},x}$.
\item[c)] for every $s\in S$, $\Supp D$ contains no associated prime of $X_s$.
\end{itemize}
If the fibers of $\pi$ are reduced, c) means that $\Supp D$ contains no irreducible
component of a fiber of $\pi$.
If $S'\to S$ is a morphism of schemes, the pullback of $D$ is a well defined effective Cartier
divisor on $X\times_S S'$, relative over $S'$. 

Let $\pi\colon X\to S$ be a flat morphism of schemes. An {\em effective $\Q$-Cartier
divisor on $X$ relative over $S$} is a formal product $xD$, where $x\ge 0$ is a rational
number and $D$ is an effective Cartier divisor on $X$ relative over $S$. If $rx\in \Z$,
we can write it as $\frac{1}{r}D'$, where $D'=rxD$ is an effective Cartier divisor on $X$ 
relative over $S$.

We will use the following special case of~\cite[Proposition 3.5]{HK04}:

\begin{lem}\label{2ext}
Let $f\colon X\to S$ be a flat morphism of schemes, whose fibers satisfy 
Serre's property $(S_2)$. Let $w\colon U\subseteq X$ be an open subset such that 
its complement $Z$ satisfies $\codim(Z_s,X_s)\ge 2$ for every $s\in S(k)$. 
Let $\cL$ be an invertible $\cO_X$-module. 
Then $\cH^i_Z(\cL)=0 \ (i=0,1)$, that is $\cL\isoto w_*(\cL|_U)$.
\end{lem}


\subsection{Log varieties, log canonical thresholds}


Let $(X/k,B)$ be a log variety. Log discrepancies in geometric valuations of $X$, and minimal 
log discrepancies in scheme points, or closed subsets of $X$, are well defined (see for example
~\cite{Amb06,Amb11}). Denote by 
$(X,B)_{lc}$ the largest open locus in $X$ where $(X,B)$ has log canonical singularities, and 
$(X,B)_{-\infty}$ its complement. For an effective $\R$-divisor $D$ on $X$ and a scheme point 
$x\in (X,B)_{lc}$, the {\em log canonical threshold at $x$ of $D$ with respect to $(X,B)$} is defined 
as
$$
\lct_x(X,B;D)=\sup\{t\ge 0;\mld_x(X,B+tD)\ge 0\}.
$$
It is $+\infty$ if $x\notin \Supp D$, $0$ if  $x\in \Supp D$ and $\mld_x(X,B)=0$, and a positive
real number if $x\in \Supp D$ and $\mld_x(X,B)>0$. It is rational if so are $B,D$ near $x$.
The reciprocal 
$$
\mu_x(X,B;D)=1/\lct_x(X,B;D)
$$ 
is called the {\em Arnold multiplicity} at $x$ of $D$ with respect to $(X,B)$. It is $0$ if 
$x\notin \Supp D$, $+\infty$ if  $x\in \Supp D$ and $\mld_x(X,B)=0$, and a positive real 
number if  $x\in \Supp D$ and $\mld_x(X,B)>0$.


\subsection{Families of log varieties}

 A {\em family of log varieties} $(X/S,B)$ consists of the following data:
\begin{itemize}
\item[a)] a flat morphism of schemes $\pi\colon X\to S$, such that $X_s$ is normal for every $s\in S$;
\item[b)] an effective $\Q$-Cartier divisor $B$ defined on $U_\pi$ and relative 
over $S$, where $w\colon U_\pi \subseteq X$ is the open locus where  the morphism $\pi$ is smooth, 
\end{itemize}
satisfying the following property: there exists an integer $r\ge 1$ such that $rB$ is Cartier and the $\cO_X$-module
$w_*((\Omega^{top}_{U/S})^{\otimes r}(rB))$ is locally free. The smallest
$r$ with this property is called the {\em index} of the family.

Here $\Omega^{top}_{U/S}$ is the top exterior product of $\Omega^1_{X/S}$, corresponding to 
the locally constant dimension of the fibers.
Property a) implies that $Z=X\setminus U$ is the union of the singular locus
of $X_s$, after all $s\in S(k)$. In particular, $\codim(Z_s,X_s)\ge 2$ for every $s\in S(k)$.
Recall that normality of a scheme is defined locally, so it does not imply irreducibility.
The fibers of $\pi$ are normal if and only if $X_s$ is normal for every $s\in S(k)$. Even if
the latter are irreducible, some fibers over non-closed points may be reducible.

Consider a family of log varieties of index $r$. For $l\in r\Z$, $lB$ is a Cartier divisor on $U$, and
we can define
$$
\omega^{[l]}=w_*((\Omega^{top}_{U/S})^{\otimes l}\otimes\cO_U(lB)).
$$ 
The $\cO_X$-module $\omega^{[l]}$ is coherent  (EGA IV, Proposition 5.11.1). By assumption, $\omega^{[r]}$ 
is an invertible $\cO_X$-module.

\begin{lem} $(\omega^{[r]})^{\otimes \frac{l}{r}}\isoto \omega^{[l]}$ for every $l\in r\Z$. 
\end{lem}

\begin{proof} We have natural multiplication maps
$
\omega^{[l]}\otimes \omega^{[l']}\to \omega^{[l+l']}.
$
By Lemma~\ref{2ext}, $\omega^{[0]}=\cO_X$. Therefore suffices to show that
$\omega^{[r]}\otimes \omega^{[l]}\to \omega^{[r+l]}$ is an isomorphism. Indeed,
denote $\cF=\omega^{[r]}\otimes \omega^{[l]}$. Our homomorphism factors as 
$$
\cF\to w_*(\cF|_U)\isoto \omega^{[r+l]}.
$$
We have $\cH^i_Z(\cF)=0\ (i=0,1)$, since $\omega^{[l]}$ satisfies this property by
Lemma~\ref{2ext}, and $\omega^{[r]}$ is locally trivial. Therefore the first map is also an isomorphism.
\end{proof}

\begin{lem}
Let $(X/S,B)$ be a family of log varieties of index $r$. Let $g\colon S'\to S$
be a morphism of schemes. Consider the induced base change diagram
\[ \xymatrix{
X \ar[d]_\pi & X'\ar[d]^{\pi'}  \ar[l]_G \\
S           & S' \ar[l]^g
} \]
Then $U_{\pi'}=G^{-1}(U_\pi)$, $(X'/S',(G|_{U_{\pi'}})^*B)$ is a family of log varieties over $S'$, and
we have natural isomorphisms
$$
G^*(\omega^{[l]} )\isoto {\omega'}^{[l]} \ (l\in r\Z).
$$
\end{lem}

\begin{proof} We have $G^*(\Omega^1_{X/S})\isoto \Omega^1_{X'/S'}$. Therefore 
$U_{\pi'}=G^{-1}(U_\pi)$ and we have a base change diagram
\[ \xymatrix{
U \ar[d]_w & U'\ar[d]^{w'}  \ar[l]_{G|_U} \\
X           & X' \ar[l]^G
} \]
For $l\in r\Z$, we obtain a natural homomorphism
$$
G^*(\omega^{[l]})=G^*w_*((\Omega^{top}_{U/S})^{\otimes l}(lB))\to
w'_*((G|_U)^*(\Omega^{top}_{U/S})^{\otimes l}(lB)) )
\isoto w'_*(\Omega^{top}_{U'/S'})^{\otimes l}(lB')))={\omega'}^{[l]}.
$$
If we denote this homomorphism by $\cF' \to \cG'$, it factors as 
$\cF' \to w'_*(\cF'|_{U'})\isoto \cG'$. Since $\omega^{[l]}$ is locally trivial,
so is $\cF'$. By Lemma~\ref{2ext}, $\cH^i_{Z'}(\cF')=0\ (i=0,1)$. 
Therefore $\cF' \to w'_*(\cF'|_{U'})$ is also an isomorphism.
\end{proof}

In particular, for every $s\in S(k)$, the fiber $X_s$ is a normal variety,
$rB_s$ is an effective Cartier divisor on $U_s=X_s\setminus \Sing(X_s)$, and if 
$r\bar{B}_s$ is the effective Weil divisor which is the closure of $rB_s$ in $X_s$, we have
a base change diagram
 \[ \xymatrix{
U_s \ar[d]_{w_s} \ar@{^{(}->}[rr] & & U \ar[d]^w \\
X_s  \ar@{^{(}->}[rr] & & X 
} \]
and an isomorphism
$$
\omega^{[r]}|_{X_s}\isoto {w_s}_*((\Omega^{top}_{U_s/k})^{\otimes r}(rB_s))=\cO_{X_s}(rK_{X_s}+r\bar{B}_s).
$$ 
Therefore $rK_{X_s}+r\bar{B}_s$ is Cartier, so that $(X_s,\bar{B}_s)$ is a log variety.
We think of $(X_s,\bar{B}_s) \ (s\in S(k))$ as an algebraic family of log varieties parametrized
by $S$. The boundary coefficients may vary.
To simplify notation, we denote $\bar{B}_s$ by $B_s$.

\begin{exmps}
\begin{itemize}
\item[1)] Let $\pi\colon X\to S$ be a smooth morphism of schemes. Then 
$(X/S,B)$ is a family of log varieties over $S$ if and only if $B=\frac{1}{r}D$,
for some integer $r\ge 1$ and an effective Cartier divisor $D$ on $X$ which 
is relative over $S$. The fibers are the log varieties $(X_s,\frac{1}{r}D|_{X_s})\ (s\in S(k))$.
\item[2)] Let $(F,B_F)$ be a log variety over $k$. Let $S$ be  a scheme. 
Denote by $B_F\times S$ the $\Q$-Cartier divisor $p_1^*(B_{F^0})$, 
defined on $F^0\times S$, where $F^0$ is the smooth locus of $F$. 
Then $(F\times S/S, B_F\times S)$ is the trivial family of log varieties over $S$,
with constant fiber $(F,B_F)$.
\item[3)] Let $(X/S,B)$ be a family of log varieties. Let $D$ be an effective Cartier
divisor on $X$, relative over $S$. Then $(X/S,B+t D)$ is a family of log divisors for
every rational $t\ge 0$. 
\item[4)] A family of log curves consists of a smooth morphism $\pi\colon X\to S$ of relative
dimension one, endowed with an effective $\Q$-Cartier divisor $B$, relative over $S$. If $r$ is the
index of the family, then $rB=D$ is an effective Cartier divisor on $X$ which is finite flat
over $S$. The $\cO_S$-module $\pi_*\cO_D$ is locally free.
\end{itemize}
\end{exmps}

\begin{lem}\label{famdes}
Let $\pi\colon X\to S$ be a flat morphism of schemes with normal fibers, and 
$S$ smooth over $k$. Then $(X/S,B)$ is a family of log varieties
if and only if $(X/k,B)$ is a log variety and the boundary supports no
irreducible components of fibers of $\pi$. Moreover, they have the same index $r$, 
and 
$
\pi^*\omega_{S/k}^{\otimes r}\otimes \omega^{[r]}_{(X/S,B)}\simeq \omega^{[r]}_{(X/k,B)}.
$
\end{lem}

\begin{proof} Since $S$ and the fibers are normal, so is $X$.
Let $U=U_\pi$ be the smooth locus of $\pi$. It follows that $\codim(X\setminus U,X)\ge 2$.
In particular, $lB$ is Cartier on $U$ if and only if $l\bar{B}$ is a Weil divisor on $X$. 
We have a short exact sequence
$$
0\to \pi^*\Omega^1_{S/k}\to \Omega^1_{U/k}\to \Omega^1_{U/S}\to 0.
$$
It induces an isomorphism $\pi^*\Omega^{top}_{S/k}\otimes \Omega^{top}_{U/S}\simeq \Omega^{top}_{U/k}$.
If $lB$ is Cartier on $U$, we obtain an isomorphism 
$$
\pi^*\omega_{S/k}^{\otimes l}\otimes \omega^{[l]}_{(X/S,B)}\simeq \omega^{[l]}_{(X/k,B)}.
$$
Therefore $\omega^{[l]}_{(X/S,B)}$ is locally free if and only if $\omega^{[l]}_{(X/k,B)}$ is locally free,
and the first claim follows. The second follows from the first.
\end{proof}

\begin{lem}
Let $(X,B)$ be a log variety. Let $f\colon X\to S$ be a morphism, with $S$ a reduced scheme.
Then there exists an open subset $\emptyset \ne V\subseteq S$ such that $(X,B)|_{f^{-1}V}\to V$
is a family of log varieties.
\end{lem}

\begin{proof} We may shrink $S$ to an open subset, so that $S$ is smooth, and $f$ is flat with
normal fibers. Let $U\subseteq X$ be the smooth locus of $f$.
Let $r\ge 1$ such that $rK_X+rB$ is Cartier. We may further shrink $S$ so that
the effective Cartier divisor $rB|_U$ becomes flat over $S$.
By Lemma~\ref{famdes}, $(X/S,B|_U)$ is a family of log varieties.
\end{proof}

\begin{lem}\label{ian}
Let $\pi\colon (X,B)\to S$ be a family of log varieties. Let $S$ be regular at 
$s$, and choose a regular system of parameters $(h_i)_i$ for $\cO_{S,s}$. The following are 
equivalent for $x\in X_s$:
\begin{itemize}
\item[a)] the fiber $(X_s,B_s)$ has log canonical singularities near $x$;
\item[b)] the log variety $(X,B+\pi^*\Sigma_s)$ has log canonical singularities near $x$,
where $\Sigma_s=\sum_i \dv(h_i)$.
\end{itemize}
\end{lem}

\begin{proof}
Let $S'=\dv(h_j)$. It is defined locally near $s$, but we may shrink $S$ to a neighborhood
of $s$. Let $X'=\pi^*(S')$, $\pi'\colon X'\to S'$ the induced morphism, and $B'=B|_{U_{\pi'}}$.
The base change data $(X',B')\to S'\ni s,(h_i|_{S'})_{i\ne j}$ satisfy the same properties.
Now $(X,B+X')$ is a log variety with lc center $X'$, a normal Cartier divisor in $X$. Therefore
the different is zero, and the codimension one adjunction formula is
$$
(K_X+B+X')|_{X'}=K_{X'}+B'.
$$
If we denote $\Sigma'_s=\sum_{i\ne j}\dv(h_i)$, we obtain 
$
(K_X+B+\pi^*\Sigma_s)|_{X'}=K_{X'}+B'+{\pi'}^*\Sigma'_s.
$
By Inversion of Adjunction~\cite{Kwk08}, $(X,B+\pi^*\Sigma_s)$ has log canonical
singularities near $x$ if and only if  $(X',B'+{\pi'}^*{\Sigma'_s})$ has log canonical
singularities near $x$. Iterating this argument proves the equivalence.
\end{proof}

\begin{thm}\label{OpenLc}
Let $\pi\colon (X,B)\to S$ be a family of log varieties. Then the locus of points $x\in X$ 
where $(X_{\pi(x)},B_{\pi(x)})$ has log canonical singularities is open.
\end{thm}

\begin{proof} We have to show that the complement
$
Z(\pi)=\cup_{s\in S(k)} (X_s,B_s)_{-\infty}
$
is closed in $X$. For this, it suffices to construct a non-empty open subset 
$V\subseteq S$ such that $Z(\pi)|\pi^{-1}(V)$ is closed.
Indeed, replacing $S$ by $S'=S\setminus V$ and $\pi$ by the induced
by base change family, we have 
$$
Z(\pi)=(Z(\pi)|\pi^{-1}(V)) \cup Z(\pi|{S'}).
$$ 
By noetherian induction, $Z(\pi|{S'})$ is closed. Therefore $Z(\pi)$ is closed.

It remains to prove the claim. For this, we may base change with the reduced structure
on $S$ and then restrict to the regular locus. Therefore $S$ is regular. 
By Lemma~\ref{famdes}, $(X,B)$ is a log variety.
We show that we have an inclusion $(X,B)_{-\infty}\subseteq Z(\pi)$, which is an equality
over some non-empty open subset of $S$. 

For the inclusion, let $x\in X\setminus Z(\pi)$, so that $(X_{\pi(x)},B_{\pi(x)})$ 
has log canonical singularities at $x$. By Lemma~\ref{ian}, $(X,B+\pi^*\Sigma_s)$ 
has log canonical singularities at $x$.
Therefore $(X,B)$ has log canonical singularities at $x$, that is $x\notin (X,B)_{-\infty}$.

By Hironaka, there exists a desingularization $\mu\colon X'\to X$ and a 
normal crossing divisor $\sum_i E_i$ which supports $B'=\mu^*(K_X+B)-K_{X'}$. 
After shrinking $S$ to an open subset, the morphism $(X',\sum_i E_i)\to S$ becomes log 
smooth. In this case, we show that the inclusion is an equality.
Assuming $(X,B)$ has log canonical singularities at $x$, we have to show that $(X_s,B_s)$
has log canonical singularities at $x$, where $s=\pi(x)$. We may shrink $X$ to a neighborhood of $x$, 
and suppose $(X,B)$ has log canonical singularities. That is the coefficients of $B'$ are at most $1$.
By adjunction, we have
$$
\mu_s^*(K_{X_s}+B_s)=K_{X'_s}+B'_s,
$$
where $B'_s=\sum_i b_i E_i|_{X'_s}$. Since $(X',\sum_i E_i)\to S$ is log smooth, a prime
divisor on $X'_s$ is contained in at most one $E_i$. Therefore the coefficients of $B'_s$
are some of the $b_i$'s, so at most $1$. Therefore $(X_s,B_s)$ has log canonical singularities at $x$.
\end{proof}

\begin{prop}\label{LctVar}
Let $\pi\colon (X,B)\to S$ be a real family of log varieties whose fibers have at most log canonical singularities.
Let $D$ be an effective Cartier divisor on $X$, relative over $S$.
Then the function $X\ni x \mapsto \lct_x(X_{\pi(x)},B_{\pi(x)};D_{\pi(x)})$ is 
lower semi-continuous and takes only finitely many values.
\end{prop}

\begin{proof} Fix $t\ge 0$. Then $\lct_x(X_{\pi(x)},B_{\pi(x)};D_{\pi(x)})\ge t$ if and only if 
$(X_{\pi(x)},B_{\pi(x)}+tD_{\pi(x)})$ has log canonical singularities at $x$.
Since $(X,B+tD|_U)\to S$ is a family of log varieties, the locus of such $x$ is
open by Theorem~\ref{OpenLc}. Therefore the function is lower semi-continuous.

To show that the function takes only finitely many values, it suffices to prove this holds over
some open non-empty subset of $S$ (by noetherian induction). Then we may assume $S$
is reduced and regular. Let $X'\to X$ be a desingularization, endowed with a 
normal crossing divisor $\sum_i E_i$ which supports both $B'=\mu^*(K_X+B)-K_{X'}$
and $D'=\mu^*D$. We finally shrink $S$ to an open subset such that $(X,\sum_i E_i)\to S$
becomes log smooth. We have 
$$
\mu_s^*(K_{X_s}+B_s+tD_s)=K_{X'_s}+B'_s+tD'_s,
$$
and the coefficients of $B'_s+tD'_s$ are some of the $b'_i+td'_i$. Therefore each
$\lct_x(X_s,B_s;D_s)$ is either $+\infty$, or the largest $t\ge 0$ such that 
$b'_i+td'_i\le 1$ for certain $i$. They belong to a finite set.
\end{proof}

\begin{cor}\label{8.5}
Let $\pi\colon (X,B)\to S$ be a real family of log varieties. Let 
$Z\subseteq X$ be a closed subset, such that $Z\to S$ is proper surjective, and 
$(X_s,B_s)$ has log canonical singularities near $Z_s$ for every $s\in S$.
Let $D$ be an effective Cartier divisor on $X$, relative over $S$.
Then the function $S\ni s\mapsto \lct_{Z_s}(X_s,B_s;D_s)$ is 
lower semi-continuous and takes only finitely many values.
\end{cor}

\begin{proof} By Theorem~\ref{OpenLc}, we may shrink $X$ near $Z$, so that the fibers of 
$\pi$ have log canonical singularities.
Denote $\gamma(x)=\lct_x(X_{\pi(x)},B_{\pi(x)};D_{\pi(x)})$.
Then $\lct_{Z_s}(X_s,B_s;D_s)=\min_{x\in Z_s}\gamma(x)$, and 
$$
\{s\in S;\lct_{Z_s}(X_s,B_s;D_s)<t\}=\pi(Z\cap \{x\in X;\gamma(x)<t\}).
$$ 
So it follows from Proposition~\ref{LctVar}.
\end{proof}


\section{Lct-variation in a linear system}


Let $(X/k,B)$ be a geometric log variety. Let $\Lambda$ be a non-empty, finite dimensional
linear system on $X$. That is $\cL$ is an invertible $\cO_X$-module, $V\subseteq \Gamma(X,\cL)$
is a non-zero finite dimensional $k$-vector subspace, and $\Lambda$ is the family of 
divisors of zeros of sections in $V$. Using a basis of $V$, we may identify $\Lambda$ with $\bP^n_k$,
where $n$ is the dimension of $\Lambda$. Inside $X\times \Lambda$ we have the universal divisor 
$H$, given by $\sum_{i=0}^ns_i(x)\lambda_i=0$, where $s_0,\ldots,s_n$ is basis of $V$ over $k$.

\begin{prop}\label{8.7} The function $(X,B)_{lc} \times \Lambda\to [0,\infty],
(P,D)\mapsto \lct_P(X,B;D)$ is lower semicontinuous and takes only finitely
many values.
\end{prop}

\begin{proof} Denote $\cX=X\times X\times \Lambda$, $S=X\times \Lambda$,
and $\pi=p_{23}\colon \cX\to S$ the projection on the last two factors.
Let $\sigma$ be the section of $\pi$ which is the diagonal on $X$ and 
the identity on $\Lambda$. Let $\cB=p_1^*B$ and $\cH=p_{13}^*H$.
Then $\pi\colon (\cX,\cB)\to S$ is a family of log varieties, $\sigma$ is a section of 
$\pi$, and $\cH$ is an effective Cartier divisor relative over $S$. 
For $s=(P,D)\in S$, the fiber $(\cX_s,\cB_s+t\cH_s)$ is isomorphic to $(X,B+tD)$
and $\sigma(s)\in \cX_s$ corresponds to $P\in X$. If we restrict the family
to $(X,B)_{lc}\times \Lambda$, it becomes log canonical, and 
$\lct_{\sigma(s)}(\cX_s,\cB_s;\cH_s)=\lct_P(X,B;D)$. We conclude by Proposition~\ref{LctVar}.
\end{proof}

\begin{exmp}
$\lct_P(X,B;D)<+\infty$ if and only if $(P,D)\in H$.
\end{exmp}

\begin{exmp}
Endow the affine line $\bA^1_k$ with a boundary $B=\sum_P b_P P$,
where $b_P\in [0,1]$, and only finitely many are non-zero.
Let $f_0,\ldots,f_n\in k[t]$ be polynomials, linearly 
independent over $k$. They induce a linear system 
$\Lambda=\{D_\lambda=\text{div}(\sum_i\lambda_if_i);\lambda\in \bP^n\}$.
Denote by $\partial$ the canonical derivation of $k[t]$.
Then $\lct_P(\bA^1,B;D_\lambda)<t$ if and only if 
$\sum_i \frac{\partial^mf_i(P)}{m!}\lambda_i=0$ for every integer
$1\le m\le \frac{1-b_P}{t}$.
\end{exmp}

\begin{thm}
Let $x\in (X,B)_{lc}$ be a scheme point. The function $\lct_x(X,B;\cdot)\colon \Lambda\to [0,+\infty]$ 
is lower semicontinuous and takes only finitely many values.
\end{thm}

\begin{proof} We have $\lct_x(X,B;D)=\max_{P\in \bar{x}}\lct_P(X,B;D)$. By Proposition~\ref{8.7},
it takes only finitely many values. For $t>0$, we have 
$$
\{D\in \Lambda;\lct_x(X,B;D)<t\}=\cap_{P\in \bar{x}}\{D\in \Lambda; \lct_P(X,B;D)<t\}.
$$
Each term on the right hand side is closed, by Proposition~\ref{8.7}. Therefore the left hand side
is also closed.
\end{proof}

In particular, $\lct_x(X,B;\cdot)$ attains its maximum (resp. minimum)
on a dense open (resp. special closed) subset of $\Lambda$. Define 
$$
\gamma_x(X,B;\Lambda)=\min_{D\in \Lambda} \lct_x(X,B;D).
$$
Denote $\mu_x(X,B;\Lambda)=1/\gamma_x(X,B;\Lambda)$, so that 
$
\mu_x(X,B;\Lambda)=\max_{D\in \Lambda} \mu_x(X,B;D).
$

\begin{thm}\label{3.5}
The function $(X,B)_{lc}\to [0,+\infty], P\mapsto \gamma_P(X,B;\Lambda)$ is lower semicontinuous
and takes only finitely many values.
\end{thm}

\begin{proof} 
By Proposition~\ref{8.7}, the values belong to a finite set. And 
$\{P\in (X,B)_{lc};\gamma_P(X,B;\Lambda)<t\}$ is the projection on the first factor
of $\{(P,D)\in (X,B)_{lc}\times \Lambda ;\gamma_P(X,B;D)<t\}$. The latter is closed
by Proposition~\ref{8.7}, and since $\Lambda$ is proper over $k$, it follows that our
level set is closed.
\end{proof}

\begin{thm}
Let $Z\subseteq X$ be a closed subset such that $Z/k$ is proper and 
$(X,B)$ has log canonical singularities near $Z$. 
The function $\Lambda\to [0,\infty], D\mapsto \lct_Z(X,B;D)$ is lower 
semicontinuous and takes only finitely many values.
\end{thm}

\begin{proof} Denote $\cX=X\times_k \Lambda$, $S=\Lambda$,
and $\pi\colon \cX\to S$ the second projection. Let $\cB=p_1^*B$
and $\cH\subset X\times \Lambda$ the universal divisor.
 Then $\pi\colon (\cX,\cB)\to S$ is a family of log varieties, $\cH$ is an effective 
 Cartier divisor relative over $S$, $\cZ=Z\times S$ is a closed subset of $X$ which
 is proper over $S$. For $s=[D]\in S$, the fiber $(\cX_s,\cB_s+t\cH_s)$ is isomorphic to $(X,B+tD)$
 and $\lct_{\cZ_s}(\cX_s,\cB_s;\cH_s)=\lct_Z(X,B;D)$. We conclude by Corollary~\ref{8.5}.
\end{proof}

In particular, $\lct_Z(X,B;\cdot)$ attains the maximal 
(resp. minimal) value on a dense open (resp. special closed) subset of $\Lambda$. Define
$$
\gamma_Z(X,B;\Lambda)=\min_{D\in \Lambda} \lct_Z(X,B;D).
$$
If $X/k$ is proper, denote $\gamma_X(X,B;\Lambda)$ by $\gamma(X,B;\Lambda)$.
Define similarly $\mu_Z(X,B;\Lambda)$ and $\mu(X,B;\Lambda)$

\begin{thm}\label{pt}
Let $(X_i/k,B_i)$ be finitely many proper log varieties, with log canonical singularities.
Let $|L_i|$ be non-empty complete linear systems on $X_i$.
Let $X=\prod_i X_i$, $B=\sum_i p_i^*(B_i)$, $L=\sum_i p_i^*(L_i)$.
Then the product log variety $(X,B)$ has log canonical singularities,
the complete linear system $|L|$ is non-empty, and 
$$
\gamma(X,B;|L|)=\min_i \gamma(X_i,B_i;|L_i|).
$$
\end{thm}

\begin{proof} By induction, suffices to consider only two factors.

Let $D_i\in |L_i|$, for $i=1,2$. Let $t_i\ge 0$ be maximal such that
$(X_i,B_i+t_iD_i)$ has log canonical singularities. Set $t=\min(t_1,t_2)$
and $D=p_1^*D_1+p_2^*D_2$. Then $D\in |L|$ and $t$ is maximal
such that $(X,B+tD)$ has log canonical singularities. Therefore
$\gamma\le \lct(X,B;D)=t$. Taking minimum after all members, we obtain
$\gamma(X,B;|L|)\le \min_i \gamma(X_i,B_i;|L_i|)$.

Suppose by contradiction that $\gamma(X,B;|L|)<\min_i \gamma(X_i,B_i;|L_i|)$.
That is, there exists $D\in |L|$ and $t\le \min_i \gamma(X_i,B_i;|L_i|)$ such 
that $(X,B+tD)$ does not have log canonical singularities. Denote 
$Z=(X,B+tD)_{-\infty}$. It is a proper subset of $X$.
If we show that $Z=\pi_i^{-1}\pi_i(Z)$ for every $i$, it follows that $Z=\emptyset$,
a contradiction.

We prove the claim for the second projection $\pi\colon X\to X_2$. Choose a 
closed point $P\in X_2$ and suppose $X_P\nsubseteq Z$. We will show that 
$X_P\cap Z=\emptyset$.
Indeed, by Hironaka's flattening, there exists a desingularization $g\colon X'_2\to X$
such that the induced Cartier divisor $D'$ on $X'=X_1\times X'_2$ admits a 
decomposition $D'=D''+\pi^*(D_2)$, where $D''$ is an effective Cartier divisor on $X'$
relative over $X'_2$, and $D_2$ is an effective Cartier divisor on $X'_2$ with normal
crossing support.

Choose a point $Q\in g^{-1}(P)$. Since $D''|_{X'_Q}\in |L_1|$ and $t\le \gamma(X_1,B_1;|L_1|)$,
the log variety $(X_1,B_1+tD''|_{X_1\times Q})$ has log canonical singularities. This is the fiber
at $Q$ of the family of log varieties $(X'/X'_2,p_1^*(B_1)+tD'')$, so by Lemma~\ref{ian},
the log variety $(X',p_1^*(B_1)+tD''+{\pi'}^*\Sigma_Q)$ has log canonical singularities
near $X'_Q$, where $\Sigma_Q$ is any local divisor cut out by a regular system of parameters
of $\cO_{X'_2,Q}$.

On the other hand, $(X,B+tD)$ has log canonical singularities at the generic point of $X_P$.
Therefore $(X',B'+tD')$ has log canonical singularities at the generic point of $X'_Q$.
Let $g^*(K_{X_2}+B_2)=K_{X'_2}+B'_2$. It follows that $(X'_2,B'_2+tD_2)$ has log canonical
singularities. Therefore $B'_2+tD'_2\le \Sigma_Q$ for some choice of local parameters at $Q$.
We deduce that $(X',p_1^*(B_1)+tD''+{\pi'}^*(B'_2+tD'_2))$ has log canonical singularities near
$X'_Q$. That is $(X',B'+tD')$ has log canonical singularities near $X'_Q$.
This holds for every $Q\in g^{-1}(P)$, so we deduce that $(X',B'+tD')$ has log canonical 
singularities over an open neighborhood of $g^{-1}(P)$. Since 
$(X',B'+tD')\to (X,B+tD)$ is log crepant, it follows that $(X,B+tD)$ has log canonical singularities
near $X_P$, that is $Z\cap X_P=\emptyset$.
\end{proof}


\subsection{Graded case}

Let $L$ be a $\Q$-Cartier divisor on $X$ such that $|nL|\ne \emptyset$ for some $n\ge 1$.
Let $x\in (X,B)_{lc}$ be a scheme point. The {\em $\alpha$-invariant at $x$ of $L$ with respect
to $(X,B)$} is defined as 
$$
\gamma_x(X,B;L)=\inf\{\lct_x(X,B;\frac{D_n}{n});n\ge 1,D_n\in |nL| \}
$$

\begin{question}\label{grsc}
Is $(X,B)_{lc}\ni x\mapsto \gamma_x(X,B;L)$ is lower semicontinuous, with finitely many rational
values?
\end{question}

Let $m\ge 1$ such that $mL$ is Cartier and $|mL|\ne \emptyset$. Then 
$\gamma_x(X,B;L)=\inf_{m\mid n} n\gamma_x(X,B;|nL|)$.

\begin{lem}\label{cga}
If $\dim X=1$, $\deg L>0$ and $x\in X$ is a closed point, then 
$
\gamma_x(X,B;L)=\frac{1-b_x}{\deg L}.
$
\end{lem}

\begin{proof} We may suppose $b_x<1$. Let $mL$ be Cartier and $|mL|\ne \emptyset$.
Let $g$ be the genus of $C$. Recall that any complete linear system of degree $g$ is
non-empty. Let $m\mid n$ and $n\deg L>g$. Then $|nL-(n\deg L-g)x|\ne \emptyset$. Therefore
$$
n\frac{\deg L}{1-b_x}-\frac{g}{1-b_x}\le \mu_x(X,B;|nL|)\le n\frac{\deg L}{1-b_x}.
$$
Dividing by $n$ and letting $n\to \infty$, we obtain the claim.
\end{proof}

\begin{prop}[\cite{Kol97}, Theorem 6.7.1]\label{Sl}
Suppose $X^d$ is proper and $L$ is a nef and big $\Q$-divisor. Then 
$\gamma_P(X,B;L)\cdot \sqrt[d]{(L^d)}\le d$ for every $P\in (X,B)_{lc}$.
\end{prop}

\begin{proof} We may scale $L$ and suppose it is Cartier. Let $0<c<\sqrt[d]{(L^d)}$ be 
a rational number. Since $h^0(nL)=(L^d)\frac{n^d}{d!}+O(n^{d-1})$ and 
$\binom{nc+d}{d}=c^d\frac{n^d}{d!}+O(n^{d-1})$, there exists an integer $n\ge 1$
such that $nc\in \Z$ and $h^0(nL)>\binom{nc+d}{d}$.

Let $Q\in X\setminus (\Sing X\cup \Supp B)$. The evaluation map $\Gamma(X,nL)\to \cO_Q/\fm_Q^{nc+1}$
has nontrivial kernel, by dimension count. Therefore there exists $D\in |nL|$ such that 
$\mult_Q(D)>nc$. Let $(X,B+\gamma D)$ be maximally log canonical at $Q$.
Let $v$ be the valuation induced by the exceptional divisor of the blow-up of $Q\in X$.
Then $0\le a(v;X,B+\gamma D)=d-\gamma \mult_v(D)<d-\gamma nc$. Therefore 
$\gamma<\frac{d}{nc}$. We conclude $\gamma_Q(X,B;|nL|)<\frac{d}{nc}$.

Since the points $Q$ are dense in $X$, Theorem~\ref{3.5} gives 
$\gamma_P(X,B;|nL|)<\frac{d}{nc}$ for every $P\in (X,B)_{lc}$.
Then $\gamma_P(X,B;L)\le n\gamma_P(X,B;|nL|)<\frac{d}{c}$ for every $P\in (X,B)_{lc}$.
Letting $c$ converge to $\sqrt[d]{(L^d)}$, we obtain the claim.
\end{proof}

If $X/k$ is proper and $(X,B)$ has log canonical singularities, 
define the {\em $\alpha$-invariant of $L$ with respect to $(X,B)$} as 
$$
\gamma(X,B;L)=\inf\{\lct(X,B;\frac{D_n}{n});n\ge 1,D_n\in |nL| \}.
$$

\begin{question}
Does $\mld(X,B)>0$ imply $\gamma(X,B;L)>0$?
\end{question}

For example, $\gamma(X,B;L)=\frac{\mld(X,B)}{\deg L}$ if $X$ is a curve and 
$L\not\sim_\Q 0$ (by Lemma~\ref{cga}).
If $X$ is smooth, $B=0$, and $A$ is a very ample divisor on $X$, the argument of~\cite[Corollary 5.11]{Vie95} 
shows that $\gamma(X,0;L)\ge \frac{1}{(L\cdot A^{d-1})}$. Theorem~\ref{pt} gives

\begin{cor}\label{cpt}
Let $(X_i/k,B_i)$ be finitely many proper log varieties, with log canonical singularities.
Let $L_i$ be $\Q$-Cartier divisors such that $|mL_i|\ne \emptyset$ for some $m\ge 1$.
Let $X=\prod_i X_i$, $B=\sum_i p_i^*(B_i)$, $L=\sum_i p_i^*(L_i)$.
Then the product log variety $(X,B)$ has log canonical singularities and 
$$
\gamma(X,B;L)=\min_i \gamma(X_i,B_i;L_i).
$$
\end{cor}


\subsection{The $\alpha$-invariant in terms of width}


Let $X$ be a normal variety. Let $\Lambda$ be a non-empty, finite dimensional linear 
system on $X$. For a prime divisor $E\subset X$, define the {\em width of $\Lambda$
at $E$} as 
$$
w_E(\Lambda)=\sup\{\mult_E(D);D \in \Lambda\}.
$$
It is a non-negative integer, zero if and only if $E$ is not a fixed component of $\Lambda$
and $\phi_\Lambda(E)=\phi_\Lambda(X)$. 
If $f\colon X'\to X$ is a proper modification and $E\subset X'$ is a prime divisor, define
the width of $\Lambda$ at $E$ as $w_E(f^*\Lambda)$. It depends only on the valuation
of $X$ defined by $E$ (called {\em geometric valuation of $X$}).

Let $L$ be a $\Q$-Cartier divisor such that $mL$ is Cartier and $|mL|\ne \emptyset$ for some $m\ge 1$.
Let $E$ be a geometric valuation of $X$. Define  the {\em width of $L$ at $E$} as 
$$
w_E(L)=\sup\{\frac{w_E(|nL|)}{n};m\mid n\}.
$$
If $X$ is proper, $w_E(L)=0$ for every geometric valuation $E$ of $X$ if and only if $L\sim_\Q 0$.
If $X$ is projective of dimension $d$, $A$ is a very ample divisor on $X$, and $E$ is a prime divisor 
on $X$, then $w_E(L)\le (L\cdot A^{d-1})$. It follows that if $X$ is proper, then $w_E(L)$ is a non-negative
real number, for every geometric valuation $E$ of $X$.
By definition, the following formulas hold:

- Let $(X,B)$ be a log variety, $\Lambda$ a non-empty finite dimensional linear system on $X$.
Let $E$ be a geometric valuation of $(X,B)_{lc}$. Then 
$$
\gamma_E(X,B;\Lambda) =
\left\{
\begin{array}{ll}
+\infty & , w_E(\Lambda)=0 \\
\frac{a(E;X,B)}{w_E(\Lambda)} & ,w_E(\Lambda)>0
\end{array} \right.
$$

- Let $(X,B)$ be a proper log variety, with log canonical singularities. Let $L$ be a $\Q$-Cartier 
divisor such that $mL$ is Cartier and $|mL|\ne \emptyset$ for some $m\ge 1$. Then 
$\gamma(X,B;L)$ is the infimum of $\gamma_E(X,B;L)$ after all geometric valuations $E$
of $X$. Equivalently,
$$
\gamma(X,B;L) =
\left\{
\begin{array}{ll}
+\infty & , L\sim_\Q 0 \\
\inf_{w_E(L)>0} \frac{a(E;X,B)}{w_E(L)} & ,L\not\sim_\Q 0
\end{array} \right.
$$


\section{The $\alpha$-invariant on toric varieties}


Let $X/k$ be a proper toric variety, let $B$ be an effective $\Q$-divisor which is 
torus invariant, such that $K_X+B$ is $\Q$-Cartier and $(X,B)$ has at most log
canonical singularities. Due to the existence of log resolutions in the toric category,
the latter condition is equivalent to $B=\sum_i b_i E_i$, where $E_i$ are the 
torus invariant prime divisors of $X$ and $b_i\in [0,1]\cap \Q$. We have 
$a(E_i;X,B)=1-b_i$.

Let $L=\sum_i l_i E_i$ be a torus invariant $\Q$-Cartier divisor. Recall that 
$
\Gamma(X,L)=\oplus_{m\in M\cap \square_L}k \cdot \chi^m,
$
where $\square_L=\cap_i\{m\in M_\R;\langle m,e_i\rangle+l_i\ge 0\}$ and $\{e_i\}=\Delta_X(1)$.

Let $V\subseteq \Gamma(X,L)$ be a non-zero, torus invariant $k$-vector subspace. 
There exists a finite set $A\subseteq M\cap \square_L$ such that $V=\oplus_{m\in A}k \cdot \chi^m$.
Let $\Lambda=\{(f)+L;f\in V\setminus 0\}$ be the corresponding linear system. 

\begin{lem} $w_{E_i}(\Lambda)$ is attained by a torus invariant member, computed by the formula:
$$
w_{E_i}(\Lambda)=\max_{m\in A}\langle m,e_i\rangle+l_i.
$$
\end{lem}

\begin{proof}
Let $t\ge 0$. The set $\{f\in V;\mult_{E_i}((f)+L)\ge t\}$ is a torus invariant vector subspace
of $V$. So it is non-zero if and only if it contains $\chi^m$ for some $m\in A$. It follows that 
the maximal (also minimal) value among $\mult_{E_i}(D)\ (D\in \Lambda)$ is attained within 
the subset $\mult_{E_i}((\chi^m)+L)=\langle m,e_i\rangle+l_i \ (m\in A)$.
\end{proof}

We have $w_{E_i}(\Lambda)=0$ if and only if $A$ is contained in the hyperplane
$\langle \cdot ,e_i\rangle+l_i=0$. It follows that we can compute $\gamma_x(X,B;\Lambda)$
for every torus invariant codimension one point $x\in X$. Indeed, $x$ is the generic point of
some $E_i$, and $\gamma_{E_i}(X,B;\Lambda)$ is $+\infty$ if $w_{E_i}(\Lambda)=0$,
and $\frac{1-b_i}{\max_{m\in A}\langle m,e_i\rangle+l_i}$ otherwise. Proposition~\ref{mS}
states that only these valuations determine $\gamma(X,B;\Lambda)$.

\begin{prop}\label{mS} 
$
\gamma(X,B;\Lambda)=\min_i \gamma_{E_i}(X,B;\Lambda).
$
In particular, $\gamma(X,B;\Lambda)$ is attained by a torus invariant member of $\Lambda$.
\end{prop}

\begin{proof} The inequality $\le$ is clear. For the converse, let $t\le \min_i \gamma_{E_i}(X,B;\Lambda)$
and $D\in \Lambda$. We have to show that $(X,B+tD)$ has log canonical singularities.

Let $\mu\colon X'\to X$ be a toric birational modification which is an isomorphism in codimension one,
and such that $X'$ is $\Q$-factorial. The toric varieties $X,X'$ have the same invariant prime divisors,
and therefore we may pullback our data to $X'$. Therefore we may suppose $X$ is $\Q$-factorial.

The conclusion is local on $X$, so we may shrink $X$ to a torus invariant affine open neighborhood 
$U$ of a fixed point. Thus $U=T_N\emb(\sigma)$, where $\sigma$ is a simplicial cone in $N_\R$
which generates $N_\R$. Let $N'$ be the lattice generated by the primitive vectors $e_i\in N$ which 
generate the extremal rays of $\sigma$. The inclusion $N'\subseteq N$ induces a toric morphism
$\tau\colon U'\to U$ which is finite, and \'etale in codimension one. We have $(U',U'\setminus T)\simeq
(\bA^d_k,\sum_{j=1}^dH_j)$, where $H_i$ are the standard hyperplanes of the affine space.
We have $\tau^*(K_X+B+tD|_U)=K_{\bA^d_k}+\sum_{j=1}^d b_j H_j+tD'$. We may suppose
$\lfloor \tau^*L\rfloor=0$, and therefore $({\bA^d_k},\sum_{j=1}^d b_j H_j+tD'=\sum_{j=1}^d b_j H_j+t\{D'\}+
t\lfloor D'\rfloor)$ has log canonical singularities by Lemma~\ref{pc}. Therefore $(X,B+tD)$ has log 
canonical singularities on $U$.
\end{proof}

\begin{lem}\label{pc}
Let $H_i$ be the standard hyperplanes of the affine space $\bA^d_k$. Let $b_i\in [0,1]$, $t>0$, 
and $0\ne P\in k[z_1,\ldots,z_d]$ such that $\deg_{z_i}(P) \le \frac{1-b_i}{t}$ for every $1\le i\le d$.
Let $D$ be the divisor of zeros of $P$. Then $({\bA^d_k},\sum_{i=1}^d b_i H_i+tD)$ has log canonical
singularities.
\end{lem}

\begin{proof} Denote $w_i=\deg_{z_i}(P)$. Consider the product of log varieties $\prod_{i=1}^d(\bP^1,b_i\cdot 0)$. 
We have $\gamma(\bP^1,b_i\cdot 0;|\cO(w_i)|)=\frac{1-b_i}{w_i}\ge t$.
By Theorem~\ref{pt}, $\gamma(\prod_{i=1}^d(\bP^1,b_i\cdot 0); |\cO(w_1,\ldots,w_d)|)\ge t$.

Now $P$ defines a divisor $D'\in |\cO(w_1,\ldots,w_d)|$. Therefore $(\prod_i \bP^1,\boxtimes_i b_i\cdot 0+tD')$
has log canonical singularities. After restricting to the complement of $\boxtimes_i \infty$, we obtain that 
$({\bA^d_k},\sum_{i=1}^d b_i H_i+tD)$ has log canonical singularities.
\end{proof}

Suppose now that $|nL|\ne \emptyset$ for some $n\ge 1$, that is $\square_L\ne \emptyset$.
The width of $L$ in $E_i$ is computed by the formula 
$$
w_{E_i}(L)=\max_{m\in \square_L}\langle m,e_i\rangle+l_i.
$$ 
Let $r\ge 1$ be the smallest integer such that the extremal points of $r\square_L$ belong to the lattice
$M$ (i.e. $\Gamma(rL)\otimes \Gamma(nL)\to \Gamma((r+n)L)$ is surjective for $n\gg 0$). Then 
$w_{E_i}(L)=\frac{w_{E_i}(|rL|)}{r}$.

The width of $L$ in $E_i$ is zero if and only if $\square_L$ is contained in the hyperplane 
$\langle \cdot,e_i\rangle+l_i=0$. In this case, $\gamma_{E_i}(X,B;L)=+\infty$. If $w_{E_i}(L)>0$, then 
$$
\gamma_{E_i}(X,B;L)=\frac{1-b_i}{w_{E_i}(L)}.
$$

Proposition~\ref{mS} for the complete linear systems $|nL|$ gives

\begin{thm}\label{mSa} $\gamma(X,B;L)=\min_i \gamma_{E_i}(X,B;L)$. In particular, 
$\gamma(X,B;L)$ is attained by some invariant member $(\chi^m)+L \ (m\in M_\Q\cap \square_L)$.
\end{thm}

Recall the the {\em stable fixed multiplicity} of $L$ in $E_i$ is defined as 
$$
f_{E_i}(L)=\inf\{\frac{\mult_{E_i}(D_n)}{n};n\ge 1,D_n\in |nL|\}.
$$
In our toric setting, it has the following combinatorial formula:
$$
f_{E_i}(L)=\min_{m\in \square_L}\langle m,e_i\rangle+l_i.
$$ 
It is zero if and only if $E_i$ is not fixed by $|nL|$ for some $n\ge 1$. More precisely, 
$r f_{E_i}(L)=\mult_{E_i}\Fix(|rL|)$, where $r$ is defined above.

Recall that the {\em width} of a convex set $\square\subset M_\R$ along a direction $e\in N_\R\setminus 0$ 
is defined as $w(\square;e)=\sup\{\langle m',e\rangle-\langle m,e\rangle;m',m\in \square\}$. We obtain the 
identity:
$$
w(\square_L;e_i)=w_{E_i}(L)-f_{E_i}(L).
$$

\begin{cor}\label{g0}
Let $(X,B)$ be a proper log variety with log canonical singularities. Let $L$
be a torus invariant $\Q$-Cartier divisor such that the linear system $|nL|$ is 
mobile for some $n\ge 1$. Then 
$$
\gamma(X,B;L)=\sup\{t\ge 0;t(\square_L-\square_L)\subseteq \square_{-K_X-B}\}.
$$
\end{cor}

\begin{proof} We have $t(\square_L-\square_L)\subseteq \square_{-K_X-B}$ if and 
only if $t(\square_L-m)\subseteq \square_{-K_X-B}$ for every $m\in \square_L$.
Since $-K_X-B=\sum_i (1-b_i)E_i$, this is equivalent to 
$t\langle m'-m,e_i\rangle +1-b_i\ge 0$ for every invariant prime divisor $E_i\subset X$,
and $m,m'\in \square_L$.
Equivalently, $t\cdot w(\square_L;e_i)\le 1-b_i$ for every $i$.
By assumption, $f_{E_i}(L)=0$ for every $i$. That is $w(\square_L;e_i)=w_{E_i}(L)$.
The condition becomes $t\cdot w_{E_i}(L)\le 1-b_i$ for every $i$, that is 
$t\le \gamma(X,B;L)$, by Theorem~\ref{mSa}.
\end{proof}

\begin{lem}
Let $(X,B)$ be a proper toric log variety, with log canonical singularities. 
Let $L$ be a torus invariant $\Q$-Cartier divisor such that $|nL|\ne \emptyset$ for some $n\ge 1$. 
\begin{itemize}
\item[a)] $X\ni P\mapsto \gamma_P(X,B;L)$ is lower semicontinuous, and takes only finitely
many values.
\item[b)] Let $P\in X\setminus T$ be a closed point outside the torus. Let $O$ be the 
generic point of the unique torus orbit which contains $P$. Then 
$\gamma_P(X,B;L)=\gamma_O(X,B;L)$. 
\end{itemize}
\end{lem}

\begin{proof}
a) Fix $t>0$. Then $\{P\in X;\gamma_P(X,B;L)<t\}=\cup_{n\ge 1}Z(t,n)$, where 
$Z(t,n)=\cup_{D_n\in |nL|}(X,B+\frac{t}{n}D_n)_{-\infty}$. Each $Z(t,n)$ is torus invariant.
And is closed by Theorem~\ref{OpenLc} applied to the universal divisor of $|nL|$.
Since $X$ contains only finitely many closed torus invariant subsets,
$Z(t,n)$ belong to a finite set. Therefore $\cup_n Z(t,n)$ is closed in $X$. We conclude
that $P\mapsto \gamma_P(X,B;L)$ is lower semicontinuous. The function is constant
on the torus orbits, and since $X$ has only finitely many orbits, we conclude that the
function takes only finitely many values.

b) The inequality $\le $ is clear. For the converse, let $t\le \gamma_O(X,B;L)$.
Then $Z(t,n)$ is a closed torus invariant subset of $X$ which does not contain
the generic point of the orbit $O$. Therefore $Z(t,n)$ is disjoint from $O$.
Therefore $t\le \gamma_P(X,B;L)$.
\end{proof}


\section{Hensley type diophantine approximation}



\subsection{Geometry of numbers~\cite{Lek69}}

Let $V\simeq \R^d$ be a finite dimensional $\R$-vector space. Let $V^*$ be the dual vector space.
The {\em dual} of a non-empty convex set $\square\subseteq V$ is defined as 
$$
\square^*=\{v^* \in V^*;\langle v^*,v \rangle+1\ge 0\ \forall v \in \square\}.
$$
It is closed, convex subset of $V^*$, containing the origin. The Duality Theorem states that 
if $\square$ is compact convex and contains the origin in its interior, then $(\square^*)^*=\square$.

Let $\square\subset V$ be a compact convex set. Let $P\in \square$ be a point. Denote
$$
\gamma(P\in \square)=\sup\{t\ge 0;P+t(\square-\square)\subseteq \square\}.
$$
It is a well defined non-negative real number, zero if and only if $P$ does not belong to the 
relative interior of $\square$. We can think of $\gamma(\cdot \in \square)$ as a distance
function to the boundary of $\square$.

Suppose $\dim\square=\dim V$ and $P$ is an interior point 
of $\square$. For each $v\in V\setminus 0$, there exist $l^+,l^->0$ such that 
$P+l^+v,P-l^-v\in \partial \square$. The supremum of the ratio $l^+/l^-$, after directions
$v\in V\setminus 0$, is called the {\em coefficient of asymmetry of $\square$ about $P$}, 
denoted $c(P\in \square)$. We have 
$$
\gamma(P\in \square)=\frac{1}{1+c(P\in \square)}.
$$ 
By definition, $c(P\in \square)\ge 1$. Therefore $\gamma(P\in \square)\le \frac{1}{2}$,
and equality holds if and only if $\square$ is symmetric about $P$. 

If $0$ is an interior point of a compact convex set $\square$, then so is $0\in \square^*$,
and $\gamma(0\in \square)=\gamma(0\in \square^*)$.

Suppose $\square$ is a compact polytope and $P$ is an interior point. To compute
$c(P\in \square)$ it suffices to consider the directions $v$ such that $P+\R v$ contains
some vertex of $\square$. In particular, the supremum in the definition of $c(P\in \square)$
is a maximum. And if both $P$ and $\square$ are rational with respect to some lattice
$\Lambda$ with $\Lambda\otimes_\Z \R=V$, then $c(P\in \square)\in \Q$.
Fixing a line which attains the maximum and passes through a vertex of $\square$,
we can apply Caratheodory's Theorem to the other boundary point, and obtain the following statement
(called the {\em simplex trick}): there exists a simplex $S$ (with $\dim S\le \dim V$), with vertices among those of 
$\square$, containing $P$ in its relative interior, and with $c(P\in \square)\le c(P\in S)$.

If $S$ is a simplex and $P$ has barycentric coordinates $(\gamma_i)_i$ with respect to the 
vertices of $S$, then $\gamma(P\in S)=\min_i\gamma_i$.

Van der Corput's theorem gives 

\begin{thm}\label{M1T} Let $\square\subset \R^d$ be a compact convex set, containing the origin 
in the interior. Then 
$$
|\Z^d\cap \Int(\square)|\ge \gamma(0\in \square)^d\vol_{\Z^d}(\square).
$$
\end{thm}

The original Minkowski's first theorem asserts that if $\square$ is symmetric about the origin, 
that is $\gamma(0\in \square)=\frac{1}{2}$,  then $\{0\}\subsetneq \Z^d\cap \Int(\square)$ if
$\vol_{\Z^d}(\square)>2^d$.


\subsection{Diophantine approximation}

For positive integers $p,q$, define integers $u_{p,q}$ recursively as follows: $u_{1,q}=q$, 
$u_{p+1,q}=u_{p,q}(1+u_{p,q})$. The following properties hold:
\begin{itemize}
\item $q\mid u_{p,q}$ and $\gcd(1+u_{p,q},1+u_{p',q})=1$ for $p\ne p'$.
\item $\sum_{i=1}^p\frac{1}{1+u_{i,q}}=\frac{1}{q}-\frac{1}{u_{p+1,q}}$. In particular,
$\sum_{i=1}^\infty\frac{1}{1+u_{i,q}}=\frac{1}{q}$.
\item $\prod_{i=1}^p(1+u_{i,q})=\frac{u_{p+1,q}}{q}$.
\item $\prod_{i=1}^p \frac{1}{1+u_{i,q}}=1-q\sum_{i=1}^p\frac{1}{1+u_{i,q}}$.
\end{itemize}

Note that $(1+u_{p,1})_{p\ge 1}=(2,3,7,43,\ldots)$ is called the Sylvester sequence
in the literature. And $u_{p,q}$ can be expressed as a polynomial in $q$, with 
leading term $q^{2^{p-1}}$.

\begin{lem}\cite{Sou05}\label{S0} 
Let $d$ be a positive integer. Let $x_1\ge \cdots\ge x_d>0$
and $y_1\ge \cdots\ge y_d>0$ be real numbers such that $\prod_{i=1}^lx_i\ge 
\prod_{i=1}^ly_i$ for every $1\le l\le d$. Then $\sum_{i=1}^dx_i\ge 
\sum_{i=1}^dy_i$, and equality holds if and only if $x_i=y_i$ for every $i$.
\end{lem}

\begin{lem}[cf. \cite{Sou05}]\label{SL} Let $d,q$ be positive integers.
Let $x_1\ge \ldots\ge x_d>0$ be real numbers such that 
$\prod_{i=1}^l x_i\le 1-q\sum_{i=1}^l x_i$ for every $1\le l\le d$. Then 
$$
\sum_{i=1}^dx_i\le \sum_{i=1}^d\frac{1}{1+u_{i,q}},
$$
and equality holds if and only if $x_i=\frac{1}{1+u_{i,q}}$ for every $i$.
\end{lem}

\begin{proof}  
Denote $y_i=\frac{1}{1+u_{i,q}}$. We have $x_1\le 1-qx_1$, that is $x_1\le y_1$.
Therefore there exists $1\le k\le d$ maximal such that 
$\prod_{i=1}^l x_i\le \prod_{i=1}^l y_i$ for every $1\le l \le k$.

Suppose $k=d$. Lemma~\ref{S0} gives $\sum_{i=1}^dx_i\le \sum_{i=1}^dy_i$.
And equality holds if and only if $x_i=y_i$ for every $i$.

Suppose $k<d$. Then $\prod_{i=1}^{k+1} x_i> \prod_{i=1}^{k+1} y_i$. It follows 
that $\prod_{i=j}^{k+1} x_i> \prod_{i=j}^{k+1} y_i$ for $1\le j\le k+1$. 
Lemma~\ref{S0} gives $\sum_{i=1}^{k+1} x_i>\sum_{i=1}^{k+1}y_i$. Therefore 
$$
\prod_{i=1}^{k+1} x_i\le 1-q\sum_{i=1}^{k+1} x_i<1-q\sum_{i=1}^{k+1} y_i=
\prod_{i=1}^{k+1} y_i.
$$
Contradiction.
\end{proof}

\begin{lem}\label{dc}
For indeterminates $T_1,\ldots,T_d$, the following formula holds:
\begin{equation*}
\det \left(
  \begin{array}{ccccc}
   1+T_1 & 1  &  & 1\\
      1     & 1+T_2   &    & 1 \\
      \vdots   &  &  \ddots &   \vdots \\
   1 &   \ldots & 1 & 1+T_d
  \end{array} \right)
  = (1+\sum_{i=1}^d\frac{1}{T_i})\prod_{i=1}^dT_i.
\end{equation*}
\end{lem}

\begin{proof} We use induction on $d$. The case $d=1$ is clear.
Let $d\ge 2$. The determinant is of the form $c_1T_1+c_0$, where $c_0,c_1$
are polynomials in $T_2,\ldots,T_d$. The constant term is obtained by setting
$T_1=0$, and we compute $c_0=\prod_{i=2}^dT_i$. The other term is the 
difference 
\begin{equation*}
\det \left(
  \begin{array}{ccccc}
   1+T_1 & 1  &  & 1\\
      1     & 1+T_2   &    & 1 \\
      \vdots   &  &  \ddots &   \vdots \\
   1 &   \ldots & 1 & 1+T_d
  \end{array} \right)
  -
  \det \left(
  \begin{array}{ccccc}
   T_1 & 1  &  & 1\\
      1     & 1+T_2   &    & 1 \\
      \vdots   &  &  \ddots &   \vdots \\
   1 &   \ldots & 1 & 1+T_d
  \end{array} \right)
\end{equation*}
We compute the determinants using the formula by permutations $\sigma$ of $\{1,\ldots,d\}$.
If $\sigma(1)\ne 1$, the corresponding difference is zero. Therefore $c_1=(1+T_1)\det_{d-1}-T_1\det_{d-1}=\det_{d-1}$. By induction,
$c_1=\sum_{i=2}^d\prod_{j\ne i,1}T_j+\prod_{i=2}^dT_i$. 
Then the determinant is $\sum_{i=2}^d\prod_{j\ne i}T_j+\prod_{i=1}^dT_i+\prod_{j\ne 1}T_j$,
so the desired identity holds for $d$. 
\end{proof}

\begin{lem}\label{15}
Let $x_1,\ldots,x_d>0$ and $c_1,\ldots,c_d\ge 1$ such that $1-\prod_{i=1}^dx_i<\sum_{i=1}^dc_ix_i<1$. 
Then there exists $z\in \N^d\setminus 0$ such that 
$
\frac{z_j}{1+\sum_i c_iz_i}<x_j
$ 
for all $j$.
\end{lem}

\begin{proof}
Consider the convex set $U=\{z\in \R^d;\Vert A z\Vert_\infty<1\}$, where A is the $d\times d$ matrix
\begin{equation*}
\left(
  \begin{array}{ccccc}
   c_1-\frac{1}{x_1} & c_2  &  & c_d\\
    c_1     & c_2-\frac{1}{x_2}   &    & c_d \\
      \vdots   &  &  \ddots &   \vdots \\
   c_1 &   \ldots & c_{d-1} & c_d-\frac{1}{x_d}
  \end{array} \right)
\end{equation*}
and the norm $\Vert \cdot \Vert_\infty$ is the maximum absolute value of the components.
By Lemma~\ref{dc}, we compute 
$$
\det A=(-1)^d\frac{1-\sum_i c_i x_i}{\prod_i x_i}.
$$ 
By assumption, $0<|\det A|<1$. Then $\vol_{\Z^d}(U)=\frac{2^d}{|\det A|}>2^d$. The convex body
$U$ is symmetric about the origin.
By Minkowski's first theorem, there exists $0\ne z\in \Z^d\cap U$. That is $z\in \Z^d\setminus 0$
and $|\sum_i c_iz_i-\frac{z_j}{x_j}|<1$ for every $j$. 
We may suppose $\sum_i c_iz_i\ge 0$, after possibly replacing $z$ by $-z$. 

The inequality $\sum_i c_iz_i-\frac{z_j}{x_j}<1$ gives
$\frac{z_j}{x_j}>-1$, that is $z_j>-x_j$. 
Since $c_i\ge 1$, we obtain $\sum_{i=1}^dx_i<1$. In particular, $x_j<1$.
Therefore $z_j>-1$, that is $z_j\ge 0$.

The other inequality $-1<\sum_i c_iz_i-\frac{z_j}{x_j}$ is equivalent to 
$
\frac{z_j}{1+\sum_i c_iz_i}<x_j.
$
\end{proof}

The following statement is the effective version of~\cite[Lemma 2.4]{Hen83}.
The case $q=1$ was obtained in~\cite[Theorem 1.1]{Ave12}.

\begin{thm}\label{LHN}
Let $q$ be a positive integer, let $1\le c_1,\ldots,c_d\le q$.
Let $x_1\ge \ldots\ge x_d>0$ such that $\sum_{i=1}^dx_i\ge \sum_{i=1}^d\frac{q}{1+u_{i,q}}$.
Suppose $x_i\ne \frac{q}{1+u_{i,q}}$ for some $i$. Then there exists 
$z\in \N^d\setminus 0$ such that 
$
\frac{c_jz_j}{1+\sum_i c_iz_i}<x_j
$
for every $j$.
\end{thm}

\begin{proof} We may suppose $\sum_{i=1}^dx_i<1$. 
By Lemma~\ref{SL} for $(\frac{x_i}{q})_i$, there exists $1\le l\le d$ such that 
$\prod_{i=1}^lx_i>q^l(1-\sum_{i=1}^lx_i)$. In particular, 
$\prod_{i=1}^lx_i>(\prod_{i=1}^lc_i)(1-\sum_{i=1}^lx_i)$.

By Lemma~\ref{15} for $(\frac{x_i}{c_i})_i$, there exists
$z\in \N^l\setminus 0$ such that 
$
\frac{c_jz_j}{1+\sum_{i=1}^l c_iz_i}<x_j
$ 
for all $1\le j\le l$. Set $z_j=0$ for $j>l$. Then $z\in \N^d\setminus 0$ satisfies
the claim.
\end{proof}


\subsection{Upper bound for coefficient of asymmetry}


The case $q=1$ of the following result was obtained in~\cite[Theorem 2.1]{AKN}.

\begin{thm}\label{ss} Let $q$ be a positive integer. Let $N$ be a $d$-dimensional lattice
and $S$ a simplex with vertices in $N$ such that $N\cap \Int(\frac{1}{q}S)=\{0\}$. 
Then $\gamma(0\in S)\ge \frac{q}{u_{d+1,q}}$,
and equality holds if and only if there exists an isomorphism $N\simeq \Z^d$ which maps
$S$ to the convex hull of $e_0,\ldots,e_d$, where $e_1,\ldots,e_d$
is the standard basis of $\Z^d$ and $e_0=-\sum_{i=1}^d\frac{u_{d+1,q}}{1+u_{i,q}}e_i$. 
\end{thm}

\begin{proof}
Let $v_0,\ldots,v_d$ be the vertices of $S$. Let $0=\sum_{i=0}^d\gamma_iv_i$ be the 
barycentric coordinates of the origin. Suppose $\gamma_0=\min_i \gamma_i$.

Suppose by contradiction that $\gamma(0\in S)<\frac{q}{u_{d+1,q}}$. That is 
$\gamma_0<\frac{q}{u_{d+1,q}}$. Then $\sum_{i=1}^d\gamma_i>\sum_{i=1}^d\frac{q}{1+u_{i,q}}$.
By Theorem~\ref{LHN}, there exists $z\in \N^d\setminus 0$ such that 
$
\frac{qz_i}{1+\sum_i qz_i}<\gamma_i \ (1\le i\le d).
$
Set $z_0=0$ and denote $|z|=\sum_i z_i$. We have 
$$
-q\sum_i z_iv_j=\sum_i ((1+q|z|)\gamma_i-qz_i)v_i.
$$
On the right hand side, the coefficients of $v_i$ are positive, and add up to $1$. Therefore
$-q\sum_i z_iv_j\in \Int(S)$. That is $-\sum_i z_iv_i\in N\cap \Int(\frac{1}{q}S)\setminus 0$, a contradiction.

We conclude that $\gamma_0\ge \frac{q}{u_{d+1,q}}$. Suppose now that $\gamma_0=\frac{q}{u_{d+1,q}}$.
The above arguments and Theorem~\ref{LHN} give $\gamma_i=\frac{q}{1+u_{i,q}}$ for $1\le i\le d$. We obtain
the barycentric coordinates
$$
0=\frac{q}{u_{d+1,q}}v_0+\sum_{i=1}^d\frac{q}{1+u_{i,q}}v_i.
$$
In particular, $\sum_{i=1}^d\frac{q}{1+u_{i,q}}(v_i-v_0)=-v_0\in N$. Since $(1+u_{i,q})_i$ are pairwise
relatively prime, we obtain $\frac{q}{1+u_{i,q}}(v_i-v_0)\in N$ for every $i$. Since $q\mid u_{i,q}$, 
we deduce $v_i-v_0=(1+u_{i,q})w_i$ for some $w_i\in N$. The vectors $w_1,\ldots,w_d\in N$
are linearly independent, and $v_0=-q\sum_{i=1}^dw_i$. 

The inequality $\sum_{i=1}^d\frac{1}{1+u_{i,q}}<\frac{1}{q}$ implies
$
v_0+\sum_{i=1}^d(0,1]w_i\subset \Int(\frac{1}{q}S).
$
Let $x_i\in (0,1]$ such that $\sum_{i=1}^d x_iw_i\in N$. Then $v_0+\sum_{i=1}^dx_iw_i\in N\cap \Int(\frac{1}{q}S)$.
Therefore $v_0+\sum_{i=1}^dx_iw_i=0$. Therefore $x_i=1$ for all $i$. We conclude
that $w_1,\ldots,w_d$ is a basis for the lattice $N$. 

But $v_i=(1+u_{i,q})w_i-\sum_{j=1}^d qw_j=\sum_ja_{ij}w_j\ (1\le i\le d)$. 
By Lemma~\ref{dc}, we compute $\det(a_{ij})=1$.
Therefore $v_1,\ldots,v_d$ is a basis of $N$.
This induces an isomorphism $N\simeq \Z^d$ such that $v_i\ (1\le i\le d)$ correspond to the
standard basis $e_i \ (1\le i\le d)$, and $v_0$ corresponds to $-\sum_{i=1}^d\frac{u_{d+1,q}}{1+u_{i,q}}e_i$.
 \end{proof}

The following is the sharp version of~\cite[Corollary 3.2]{Hen83}:

\begin{thm}\label{sHe}
Let $\square\subset\R^d$ be a compact polytope with vertices in $\Z^d$, of dimension $d$. Suppose
$\Z^d\cap \Int\square$ has cardinality $q\ge 1$. Then for every $P\in \Z^d\cap \Int\square$
we have 
$$
\gamma(P\in \square)\ge \frac{q}{u_{d+1,q}}.
$$
\end{thm}

\begin{proof} Fix $P\in \square$. We may replace $\square$ by $\square-P$, so that $P=0$.
Suppose by contradiction that $\gamma(0\in \square)<\frac{q}{u_{d+1,q}}$.
By the simplex trick, there exists a simplex $S$ with vertices
among those of $\square$, which contains $0$ in its relative interior, and $\gamma(0\in \square)\ge \gamma(0\in S)$.

Let $\dim S=d'\le d$. Then $\gamma(0\in S)\le \gamma(0\in \square)<\frac{q}{u_{d+1,q}} \le \frac{q}{u_{d'+1,q}}$. 
By Theorem~\ref{ss}, there exists $0\ne e\in \Z^d\cap \relint(\frac{1}{q}S)$.
Then $0,e,2e,\ldots,qe$ are $q+1$ distinct lattice points in the relative interior of $S$. They must be contained
in $\Z^d\cap \Int\square$. Contradiction!
\end{proof}

As in~\cite[Theorem 3.6, Corollary 3.7]{Hen83}, we obtain

\begin{cor} Let $\Z^d,\square,q$ as above. Then $\vol_{\Z^d}(\square)\le q (\frac{u_{d+1,q}}{q})^d$ and
$\Z^d\cap \square$ has cardinality at most $d+d!q (\frac{u_{d+1,q}}{q})^d$.
\end{cor}


\subsection{Errata to~\cite{Amb08}}

The upper bound $n \le c_d q^d$ in~\cite[Theorem 1.1]{Amb08} is not correct. 
In Step 1 of the proof, the constant $\gamma$ depends not only on $d-1$, but
on $q$ as well. Since $\Lambda\simeq\Z^{d-1}$, $S$ has vertices in 
$\frac{1}{q}\Lambda$ and $\Lambda\cap \Int(S)=\{0\}$, Lemma~\ref{bl} gives 
$\gamma\ge \frac{q}{u_{d,q}}$. Step 2 of the proof gives $j\le d!q^{d-1}\gamma^{-d+1}$. 
Since $ja\in \Z$, we obtain a correct effective upper bound for Theorem 1.1
$$
n\le d! u_{d,q}^{d-1}q.
$$
This bound is probably not sharp.

\begin{lem}\label{bl}
Let $S$ be a simplex with vertices in $\frac{1}{q}\Z^d$, and such that $\{0\}=\Z^d\cap \Int(S)$.
Then $\gamma(0\in S)\ge \frac{q}{u_{d+1,q}}$.
\end{lem}

\begin{proof}
The simplex $S'=qS$ has vertices in $\Z^d$ and $\{0\}=\Z^d\cap \Int(\frac{1}{q}S')$.
By Theorem~\ref{ss}, $\gamma(0\in S')\ge \frac{q}{u_{d+1,q}}$. But 
$\gamma(0\in S')=\gamma(0\in S)$.
\end{proof}


\section{Toric log Fano varieties}


Let $(X,B)$ be a {\em toric log variety}, that is a log variety such that $X=T_N\emb(\Delta)$ 
is a toric variety and $B$ is an invariant $\Q$-Weil divisor. If $\sum_i E_i$ is the 
complement of the torus in $X$, we have $B=\sum_i(1-a_i)E_i$. Since $K_X+\sum_i E_i=0$,
we obtain 
$$
K_X+B=\sum_i -a_iE_i.
$$
Let $\sigma\in \Delta(top)$. There exists $\psi_\sigma\in M_\Q$ such that
$(\chi^{\psi_\sigma})+K_X+B$ is zero on the open subset $U_\sigma$ of $X$.
That is $\langle \psi_\sigma,e_i\rangle=a_i$ for every $e_i\in \sigma(1)$.


\subsection{Minimal log discrepancies}


Each $e\in N^{prim}\cap \sigma$ defines a toric valuation $E_e$ over $U_\sigma$,
with log discrepancy computed by the formula
$$
a(E_e;X,B)=\langle \psi_\sigma,e\rangle.
$$
Since log resolutions exist in the toric category, we obtain 
$$
\mld(U_\sigma;X,B)=\min\{\langle \psi_\sigma,e\rangle;0\ne e\in N\cap \sigma\}.
$$
The moment polytope associated to the $\Q$-divisor $-K_X-B$ is  
$$
\square_{-K_X-B}=\{m\in M_\R;\langle m,e_i\rangle+a_i\ge 0\ \forall e_i\in \Delta(1)\}.
$$
It contains the origin of $M$.
If $a_i>0$ for every $i$, denote by $P$ the convex hull of $\frac{e_i}{a_i}\ (e_i\in \Delta(1))$.
We compute $P^*=\square$. By duality,
$$
\inf\{t>0;e\in tP\}=-h_{\square}(e) \ \forall e\in N_\R. 
$$


\subsection{Toric weak log Fano varieties}


Suppose moreover that $X$ is proper and $-K_X-B$ is $\Q$-semiample.
Note that $-K_X-B$ is $\Q$-semiample if and only if it is nef. And if $a_i>0$ for every $i$, the 
$\Q$-divisor $-K_X-B$ is necessarily big.

The $\Q$-semiampleness condition is equivalent to $-\psi_\sigma\in \square$, for every $\sigma\in \Delta(top)$.
And $-K_X-B$ is $\Q$-ample if and only if the vertices of $\square$ are precisely $(-\psi_\sigma)_{\sigma\in \Delta(top)}$.
Let $e\in N^{prim}$, let $\sigma\in \Delta(top)$
contain $e$. Then $\square+\psi_\sigma\subseteq \sigma^\vee$. Therefore 
$
\langle -\psi_\sigma,e\rangle=h_\square(e).
$
Therefore
$$
a(E_e;X,B)=-h_\square(e)=-\inf_{m\in \square}\langle m,e\rangle.
$$
We obtain the global formula
$
\mld(X,B)=-\sup_{0\ne e\in N} h_\square(e).
$
By duality, 
$$
\mld(X,B)=\inf\{t>0; \{0\}\subsetneq N\cap tP\}.
$$
By Corollary~\ref{g0}, the $\alpha$-invariant of $-K_X-B$ with respect to $(X,B)$
is computed by the formula
$$
\gamma(X,B;-K_X-B)=\gamma(0\in \square).
$$
By duality, $\gamma(0\in \square)=\gamma(0\in \square^*)$. Denote $d=\dim X$.

\begin{thm}\label{GB} Let $q$ be a positive integer. If $\mld(X,B)\ge \frac{1}{q}$, then 
$\gamma(0\in \square)\ge \frac{q}{u_{d+1,q}}$.
\end{thm}

\begin{proof} The assumption gives $a_i\ge \frac{1}{q}$ for every $i$.
Suppose by contradiction that $\gamma< \frac{q}{u_{d+1,q}}$. 
By the simplex trick, there exists a simplex $S$, with vertices among those of $\square^*$, 
such that $0\in \relint S$ and $\gamma(0\in S)\le \gamma(0\in \square^*)$. Therefore 
$$
\gamma(0\in S)<\frac{q}{u_{d+1,q}}.
$$
Let $d'\le d$ be the dimension of $S$, and $\frac{e_i}{a_i}\ (0\le i\le d')$ its vertices. 
Let $\gamma_0,\ldots,\gamma_{d'}>0, \sum_{i=0}^{d'} \gamma_i=1$ and 
$
0=\sum_{i=0}^{d'}\gamma_i\frac{e_i}{a_i}.
$
We have $\gamma(0\in S)=\min_{i=0}^{d'}\gamma_i$. Say $\gamma_0$ is minimal.
Then $\gamma_0<\frac{q}{u_{d+1,q}}\le \frac{q}{u_{d'+1,q}}$, that is 
$\sum_{i=1}^{d'}\gamma_i>\sum_{i=1}^{d'}\frac{q}{u_{d'+1,q}}$. Note that $1\le qa_i\le q$. 
By Theorem~\ref{LHN}, there exists $z\in \N^{d'}\setminus 0$ such that 
$$
\frac{qa_i z_i}{1+\sum_{j=1}^{d'}qa_jz_j}<\gamma_i\ (1\le i\le d').
$$
These inequalities are equivalent to
$$
\frac{1}{q}>\max_{i=1}^{d'} \frac{a_iz_i}{\gamma_i}-\sum_{j=1}^{d'} a_jz_j.
$$
The right hand side is the smallest $t>0$ such that $e=\sum_{i=1}^{d'}-z_ie_i$
belongs to $tS$. Therefore $qe$ belongs to the relative interior of $S$, so in the
relative interior of $P$ as well.

Let $e'=e/n$ be the primitive vector on the ray generated by $e$. Then $e'$
defines a toric valuation of $X$ with log discrepancy
$a(E_{e'};X,B)<\frac{1}{qn}$. Therefore $\mld(X,B)<\frac{1}{q}$, a contradiction.
\end{proof}

\begin{cor}\label{vb} If $\mld(X,B)\ge \frac{1}{q}$, then $\sqrt[d]{(-K_X-B)^d}\le \frac{d}{q}u_{d+1,q}$.
\end{cor}

\begin{proof} By Proposition~\ref{Sl}, $\gamma\cdot \sqrt[d]{(-K_X-B)^d}\le d$.  
\end{proof}

\begin{exmp}\label{mE} 
The lower bound in Theorem~\ref{GB} is sharp. In dimension one, $(\bP^1,\frac{q-1}{q}\cdot\infty)$
is the only example which attains it. A higher dimensional
example is constructed as follows. Let $e_1,\ldots,e_d$ be the standard basis of $\Z^d$, and
$$
e_0=\sum_{i=1}^d-\frac{u_{d+1,q}}{q(1+u_{i,q})}e_i.
$$
Then $e_0,\ldots,e_d$ are primitive vectors in the lattice $N=\Z^d$, which they generate.
Set $a_0=\frac{1}{q}$ and $a_1=\cdots=a_d=1$. This defines a toric log Fano $(X^d,(1-\frac{1}{q})E_0)$ with 
\begin{align*}
\square_{-K_X-B} & = \{m\in \R_{\ge -1}^d;\sum_{i=1}^d \frac{u_{d+1,q}}{1+u_{i,q}}m_i\le 1\} \\
   & =  (-1,\ldots,-1)+\{m\in \R_{\ge 0}^d;\sum_{i=1}^d \frac{1}{1+u_{i,q}}m_i\le \frac{1}{q}\}
\end{align*}
The simplex $\square_{-K_X-B}=\Conv(v_0,\ldots,v_d)$ 
contains in interior the origin of $M$, with barycentric coordinates
$$
0=\frac{q}{u_{d+1,q}}v_0+\sum_{i=1}^d\frac{q}{1+u_{i,q}}v_i.
$$
So $\gamma(X,B;-K_X-B)=\frac{q}{u_{d+1,q}}$. One checks that $\mld(X,B)=\frac{1}{q}$,
$\Bs|-q(K+B)|=\emptyset$ and $(-q(K_X+B))^d=\frac{u_{d+1,q}}{q}$. In particular, the
upper bound in Corollary~\ref{vb} is not sharp.
\end{exmp}

\begin{thm}[cf.~\cite{BB92}]\label{FLF}
Fix $d\ge 1$ and $\epsilon\in (0,1]$. Consider toric proper varieties $X$ such that 
$\dim X=d$ and there exists an invariant effective $\Q$-divisor $B$ 
such that $-K_X-B$ is $\Q$-semiample and $\mld(X,B)\ge \epsilon$. Then $X$ 
belongs to finitely many isomorphism types.
\end{thm}

\begin{proof} Let $\gamma=\gamma(0\in \square_{-K_X-B})=\gamma(0\in P_{-K_X-B})$. By Theorem~\ref{GB},
$\gamma \ge \gamma(d,\epsilon)>0$. We have $\epsilon\gamma(P-P)\subseteq \epsilon P$.
Therefore $\{0\}=N\cap \Int(\epsilon \gamma (P-P))$. Theorem~\ref{M1T} gives
$$
\vol_N(P-P)\le \frac{1}{\gamma^d\epsilon^d}.
$$
Let $C$ be the convex hull of $\Delta(1)$. Since $C\subseteq P$, we deduce
that $\vol_N(C-C)$ is bounded above. Since $C$ is a lattice polytope, it follows that the pair
$(N,C-C)$ belongs to finitely many isomorphism types. There exist only finitely many fans $\Delta$ 
with given $\Delta(1)$. Therefore $(N,\Delta)$ belongs to finitely many isomorphism types.
\end{proof}


\subsection{Examples}

Consider toric log Fano varieties $(X,B)$, of Picard number one, and such that $X$
has no nontrivial toric finite covers which are \'etale in codimension one. 
Let $X=T_N\emb(\Delta)$, $\dim N=d$. Then $\Delta(1)=\{e_0,\ldots,e_d\}$,
where $e_0,\ldots,e_d\in N$ are primitive and generate the lattice $N$, and no
$d$ of them are linearly dependent. The log discrepancies in invariant prime divisors
are rational numbers $a_0,\ldots,a_d\in [0,1]$, not all zero. We have 
$$
0=\sum_{i=0}^dx_ie_i\ (x_i\in \Q_{>0},\sum_i x_i=1).
$$
The vector $e_i$ is primitive in the lattice generated by $e_0,\ldots,e_d$ if and only if 
$n_i=1$, where 
$$
n_i=\frac{\gcd(qx_0,\ldots,\widehat{qx_i},\ldots,qx_d)}{\gcd(qx_0,\ldots,qx_d)} \ (q\ge 1,qx\in \Z^{d+1}).
$$
Up to isomorphism, $(X,B)$ is uniquely determined by $(x_i)_i$ and $(a_i)_i$. 
The dual polytope $\square_{-K_X-B}^*$ is the simplex with vertices $e_i/a_i$.

The dual description of $(X,B)$ is: $0\in \square\subset M_\R$ is a rational simplex with vertices 
$v_0,\ldots,v_d$, and $0$ lies at distance $a_i$ to the face of $\square$ opposite $v_i$ (distance 
measured with respect to $e_i$, the primitive interior direction normal to the face opposite to $v_i$). Let 
$$
0=\sum_{i=0}^d\gamma_i v_i\ (\gamma_i\ge 0,\sum_i \gamma_i=1).
$$
Then $\gamma_i=\frac{a_i}{w_i}$, where $w_i=\width(\square;e_i)$. The above data are related as follows:

- $\langle v_i,e_j\rangle+a_j=0$ for $i\ne j$, and $\langle v_j,e_j\rangle+a_j=w_j$.

- $(\sum_i a_ix_i)(\sum_i \frac{1}{w_i})=1$.

- $x_j=\frac{\frac{1}{w_j}}{\sum_{i=0}^d\frac{1}{w_i}}$, $\gamma_j=\frac{a_jx_j}{\sum_{i=0}^da_ix_i}$,
$w_j=\frac{\sum_i a_ix_i}{x_j}$.

We have $K_X+B+\sum_{i=0}^da_iE_i\sim 0$. For $0\le i\le d$, set $D_i=(\chi^{v_i})-K_X-B=
w_iE_i$. So $K_X+B+D_i\sim_\Q 0$ and $\lct(X,B;D_i)=\gamma_i$. 

Note that $B=0$ if and only if $a_i=1$ for all $i$, and then $x_i=\gamma_i=\frac{1}{w_i}$ for all $i$.

\begin{lem}
$r(K_X+B)$ is Cartier if and only if $\Bs |-r(K_X+B)|=\emptyset$, if and only if $ra_i,rw_i\in \Z$ for every $i$.
\end{lem}

\begin{lem}
$\gamma(X,B;-K_X-B)=\min_{i=0}^d \gamma_i$.
\end{lem}

\begin{lem}\label{mldf}
Let $0<\epsilon\le \min_{i=0}^da_i$. Then $\mld(X,B)\ge \epsilon$ if and only if there does not exist
an integer $n\ge 0$ such that $\sum_i\{(n+1)x_i\}=1$, $\{(n+1)x\}\ne x$, and
$$
\sum_{i=0}^d a_i\{(n+1)x_i\}-\epsilon< \min_{j=0}^d w_j \{(n+1)x_j\}.
$$
The fractional parts of vectors are defined componentwise.
\end{lem}

\begin{proof} We clearly have $\mld(X,B)=a\le \min_{i=0}^da_i$. Log discrepancies in toric valuations
are computed as follows: if $e=\sum_{i=0}^dt_ie_i$ is a primitive vector in $N$, the induced toric
valuation $E_e$ has log discrepancy
$$
a(E_e;X,B)=\inf\{\epsilon>0;\sum_{i=0}^dt_ie_i \in \epsilon \cdot \square^*\} 
= \max_{j=0}^d \sum_{i=0}^d a_it_i-w_j t_j.
$$
We have $a\cdot \square^*\subseteq C=\Conv(e_0,\ldots,e_d)$. 
Therefore the primitive vectors which attain $a$ must be contained in $C$. But $N\cap C$
is parametrized as follows:
$$
N\cap C\setminus \{e_i;i\}=\{\sum_i \{(n+1)x_i\}e_i= \sum_i -\lfloor (n+1)x_i\rfloor e_i;
n\ge 0, \sum_i\{(n+1)x_i\}=1 \}.
$$
Note $\sum_i\{(n+1)x_i\}=1$ if and only if $\sum\lfloor (n+1)x_i\rfloor=n$. And 
$\sum_i \{(n+1)x_i\}e_i=0$ if and only if $\{(n+1)x\}\ne x$, if and only if $nx\in \Z^{d+1}$. We obtain
$$
a=\min(\min_{i=0}^da_i, \min_{\sum_i\{(n+1)x_i\}=1,\{(n+1)x\}\ne x} \sum_i a_i\{(n+1)x_i\}-\min_j w_j \{(n+1)x_j\})
$$
\end{proof}

Lemma~\ref{mldf} has a geometric reformulation: let $S=\{x\in \R_{\ge 0}^{d+1};\sum_{i=0}^dx_i=1\}$,
with vertices $P_0,\ldots,P_d$. Our point $x=\sum_i x_i P_i$ lies in the interior of $S$. 
Let $S_\epsilon(x)$ be the simplex with vertices $Q_i=(1-\frac{\epsilon}{a_i})x+\frac{\epsilon}{a_i}P_i\ (0\le i\le d)$. 
It is a neighborhood of $x$ in $S$ obtained by sliding each vertex of $S$ towards $x$, with a certain weight.
Then:
\begin{itemize}
\item $\mld(X,B)\ge \epsilon$ if and only if $\Int S_\epsilon(x)$ contains no $\{(n+1)x\}$ other than $x$.
\item $x=\sum_i \gamma_i Q_i$.
\end{itemize}
So both $\mld(X,B)$ and $\gamma(X,B;-K_X-B)$ are encoded by the neighborhood $x\in S_\epsilon(x)$.
{\em The condition $\mld(X,B)\ge \epsilon$ means that the rational point $x$ is badly approximated by the
fractional parts of its multiples, and Theorem~\ref{FLF} states that such $x$ belong to a finite set}.

\begin{lem}
Let $\epsilon\le \min_{i=0}^da_i$. Then $\mld(X,B)<\epsilon$ if and only
if one of the following equivalent conditions hold:
\begin{itemize}
\item[a)] There exists $n\ge 1$ such that $\sum_{i=0}^d\{(n+1)x_i\}=1$, $\{(n+1)x\}\ne x$, and 
$$
\frac{\sum_{i=0}^d a_i\{(n+1)x_i\}-\epsilon}{\sum_{i=0}^da_ix_i} < \min_{j=0}^d\frac{\{(n+1)x_j\}}{x_j}.
$$
\item[b)] There exists $z\in \N^{d+1}$ such that $\max_j w_jz_j-\sum_i a_iz_i\in (0,\epsilon)$.
\end{itemize}
Moreover, $z=\lfloor (1+|z|)x\rfloor$. So if $z$ exists, it is uniquely determined by $|z|$ and $x$.
If $qx\in \Z^{d+1}$, b) can be decided by considering only the multiples $x,2x,\ldots,(q-1)x$.
\end{lem}

\begin{exmp}
Suppose $B=0$. That is $a_i=1$, $x_i=\gamma_i=\frac{1}{w_i}$.
Then $\mld(X,0)<\epsilon$ if and only if there exists $n\ge 1$ such that 
$\{(n+1)x\}\ne x$ and $1-\epsilon<\min_j \frac{\{(n+1)x_j\}}{x_j}$, if and only if 
there exists $z\in \N^{d+1}\setminus 0$ such that $\frac{z_j}{\epsilon+\sum_{i=0}^dz_i}<x_j$
for every $j$.
\end{exmp}


\end{document}